\documentclass[letterpaper,11pt,reqno]{amsart}
\usepackage{graphicx,amsmath,amssymb,amsthm, float}

\usepackage{srcltx}
\usepackage{mathrsfs}

\usepackage[utf8]{inputenc}
\usepackage[T1]{fontenc}

\usepackage[left=1in,right=1in, top=1in, bottom=1in]{geometry}
\usepackage[hidelinks]{hyperref}

\newtheorem{X}{X}[section]
\newtheorem{corollary}[X]{Corollary}

\newtheorem{lemma}[X]{Lemma}

\newtheorem{proposition}[X]{Proposition}
\newtheorem{theorem}[X]{Theorem}

\renewcommand{\d}{{\rm d}}
\theoremstyle{definition}
\newtheorem{remark}[X]{Remark}

\usepackage{xcolor}

\newcommand{\Dstar}{\mathcal D_{k,N}^*(\phi;X)}

\newcommand{\eps}{\varepsilon}

\title[Extending the support in an Iwaniec--Luo--Sarnak family]{Extending the unconditional support in an Iwaniec--Luo--Sarnak family}

\author{Lucile Devin, Daniel Fiorilli and Anders S\"odergren}
\address{Univ. Littoral C\^ote d'Opale, UR 2597
	LMPA, Laboratoire de Math\'ematiques Pures et Appliqu\'ees \newline
	\rule[0ex]{0ex}{0ex}\hspace{8pt} Joseph Liouville,
	F-62100 Calais, France}
\email{lucile.devin@univ-littoral.fr}
\address{Univ. Paris-Saclay, CNRS, Laboratoire de mathématiques d'Orsay, 91405, Orsay, France} 
\email{daniel.fiorilli@universite-paris-saclay.fr}
\address{Department of Mathematical Sciences, Chalmers University of Technology and the University \newline
	\rule[0ex]{0ex}{0ex}\hspace{8pt} of Gothenburg, SE-412 96 Gothenburg, Sweden}
\email{andesod@chalmers.se} 

\date{\today}
\begin{document}

\maketitle

\begin{abstract}
We study the harmonically weighted one-level density of low-lying zeros of $L$-functions in the   family of holomorpic newforms of fixed even weight $k$ and prime level $N$ tending to infinity.
For this family, Iwaniec, Luo and Sarnak proved that the Katz--Sarnak prediction for the one-level density holds unconditionally when the support of the Fourier transform of the implied test function is contained in $(-\tfrac32,\tfrac32)$.
In this paper, we extend this admissible support to $(-\Theta_k,\Theta_k)$, where $\Theta_2 = 1.866\dots$ and $\Theta_k$ tends monotonically to $2$ as $k$ tends to infinity. This is asymptotically as good as the best known GRH result.  
The main novelty in our analysis is the use of zero-density estimates for Dirichlet $L$-functions.
\end{abstract}

\section{Introduction}

Since the seminal paper of Iwaniec, Luo and Sarnak~\cite{ILS} on low-lying zeros of families of $L$-functions attached to holomorhic modular forms, much work has been done towards the Katz--Sarnak heuristics~\cite{KS} (see for instance~\cite{SST} and the references therein). In the family of cusp forms of fixed weight $k$ and large level $N$, Iwaniec, Luo and Sarnak have proven that the Katz--Sarnak conjecture for the one-level density holds for test functions $\phi$ for which ${\rm supp}(\widehat \phi) \subset (-\frac 32,\frac 32)$ unconditionally, and ${\rm supp}(\widehat \phi) \subset (-2,2)$ under the relevant Riemann Hypothesis.  These results have been refined by Miller~\cite{M} and Miller--Montague~\cite{MM} to estimates containing lower-order terms, under the same assumptions.  In the current paper, we extend the admissible support of the test function for the harmonically weighted one-level density: we show unconditionally that the Katz--Sarnak conjecture holds
 whenever ${\rm supp}(\widehat \phi) \subset (-\Theta_k,\Theta_k)$, where $\Theta_k$ is a monotonically increasing function of $k$ such that $\Theta_2=1.866\dots$,  $\Theta_4=1.942\dots$, and $ 0<2- \Theta_k \ll \frac 1k $ for $k\geq 6$.

Before we state our main result, we need to introduce some notations. We fix  a basis $B_k^*(N)$ of Hecke eigenforms of the space $H_k^*(N)$ of newforms of prime level~$N$ and weight~$k$. We also renormalize so that for every $f(z)= \sum_{n=1}^{\infty}\lambda_f(n) n^{\frac{k-1}2} e^{2\pi i n z}\in B_k^*(N)$, we have $\lambda_f(1)=1$. We will use the harmonic weights defined as
$$\omega_f(N):= \frac{\Gamma(k-1)}{(4\pi)^{k-1} (f,f)_N}; \qquad (f,f)_N:= \int_{\Gamma_0 (N) \backslash \mathbb H} y^{k-2} |f(z)|^2 \d x\d y. $$
Note that (combine \cite[Lemma 2.5]{ILS} and \cite{HL,Iw})
 $ (kN)^{-1-\eps} \ll_{\eps} \omega_f(N) \ll_{\eps} (kN)^{-1+\eps}.$ 
Our main object of study is the one-level density, which is defined as
$$ \Dstar:= \frac{1}{\Omega_k(N)}\sum_{f \in B^*_k(N)} \omega_f(N) \sum_{\gamma_f} \phi\Big( \gamma_f \frac{\log X}{2\pi} \Big), $$
where $\rho_f=\frac 12+i\gamma_f$ runs through the non-trivial zeros of $L(s,f)$ (note that $\gamma_f$ might be non-real), $\phi$ is an even Schwartz function whose Fourier transform is compactly supported, the parameter $X=k^2N$ is the analytic conductor of $f$, and the total weight is given by
$$ \Omega_k(N) = \sum_{f \in B^*_k(N)} \omega_f(N).$$
We have the estimate 
$$\Omega_k(N) = 1+O_k(N^{-1}) $$
(see Lemma~\ref{lemma:peterssonbound}).
We can now state our main result.

\begin{theorem}
\label{theorem main}
Let $\phi$ be an even Schwartz function for which  ${\rm supp}(\widehat \phi) \subset (-\Theta_k,\Theta_k)$, where

\begin{equation}
 \Theta_k:= \begin{cases} 1 + \frac{\sqrt 3}{2} & \text{ if } k=2; \\ %+\frac{\sqrt 3}{(\sqrt 3-1)(3-\sqrt 3)} & \text{ if } k=2; \\
2\big( 1-\frac 1{10 k-5}\big) &\text{ if } k\geq 4.
\end{cases}
\label{equation definition thetak}
\end{equation}
Then, for $N$ running through the set of prime numbers, we have the estimate 
\begin{equation}
\mathcal D_{k,N}^*(\phi;X) =  \int_{\mathbb R} W(O)(x) \phi(x) \d x  +o_{N\rightarrow \infty}(1),
\label{equation main theorem}
\end{equation}
where $W(O)(x)=1+\frac 12\delta_0(x)$.
\end{theorem}

\begin{remark}
$ $
\begin{itemize}

\item
In the literature on low-lying zeros, there are already several unconditional  results with large admissible supports.
We mention the result of Drappeau--Pratt--Radziwi\l\l~\cite{DPR} in the family of Dirichlet $L$-functions and that of Fouvry--Kowalski--Michel~\cite{FKM} in the family of symmetric square $L$-functions of holomorphic cusp forms of prime level.

\item The use of zero-density estimates in order to circumvent the assumption of the Riemann Hypothesis in the context of low-lying zeros is an idea which was first hinted by Brumer~\cite{B}, and subsequently successfully used by Kowalski--Michel~\cite{KM} in families of holomorphic cusp forms and Baier--Zhao~\cite{BZ} in families of elliptic curve $L$-functions. The novelty in our approach is to use such estimates after an application of the Petersson formula, which relates low-lying zeros of cusp form $L$-functions to sums over zeros of Dirichlet $L$-functions, for which powerful zero-density estimates have been known for almost a century (see e.g.~\cite{L,Mo}).

\item Our techniques also yield unconditional extensions of the admissible support in families of Maass forms $L$-functions in the level aspect. These results will appear in a forthcoming paper.

\end{itemize}

\end{remark}

Here is a brief summary of the proof of Theorem~\ref{theorem main}. In Section~\ref{Section :: prerequisites}, we apply the explicit formula and express the one-level density $\Dstar$ as a sum over eigenvalues of Hecke operators at prime power values. Averaging over the family of newforms of prime level, we apply the Petersson formula and turn this last expression into a sum of Kloosterman sums weighted by Bessel functions. In Section~\ref{Section:Prime sum}, we rewrite the Kloosterman sums in terms of Dirichlet characters and Gauss sums using orthogonality. This last expression allows us to apply Mellin inversion and use a variant of the explicit formula for Dirichlet $L$-functions. Finally, in Section~\ref{Section:zero density} we complete the proof of Theorem~\ref{theorem main} using zero-density estimates. It is the shape of these zero-density estimates which gives the exact restriction on the support which appears in Theorem~\ref{theorem main}. Note that if one assumes the Grand Density Conjecture~\cite[p. 250]{IK}, then one can prove that the Katz--Sarnak prediciton holds in the full range $ {\rm supp}(\widehat \phi)\subset(-2,2)$, independently of $k$ (see Remark~\ref{remark grand}).

\section*{Acknowledgments}

The first author was supported by the grant KAW 2019.0517 from the Knut and Alice Wallenberg Foundation. The third author was supported by the grant 2021-04605 from the Swedish Research Council.
We thank the Anna-Greta and Holger Crafoord Fund and the Royal Swedish Academy of Sciences for supporting this project via the grant CRM2020-0008, as well as the IHP Research in Paris program for providing funding and excellent working conditions.

\section{Prerequisites and first estimates}
\label{Section :: prerequisites}

The goal of this section is to gather a few identities and estimates which will be central in our analysis, including the explicit and Petersson formulas. We will then bound some of the terms in these formulas and reduce our problem to estimates on averages of Kloosterman sums.

We begin with the explicit formula.
\begin{lemma}[Explicit Formula]
Let $\phi:\mathbb R \rightarrow \mathbb R$ be an even Schwartz  function whose Fourier transform has compact support. Then for $X > 1$ we have the formula

\begin{multline}
\Dstar = \widehat{\phi}(0) \frac{\log\frac N{\pi^2}}{\log X}+ \frac 1{\log X} \int_{\mathbb R} \Big(  \frac{\Gamma'}{\Gamma} \Big( \frac 14+ \frac{k+ 1}4 +\frac{\pi i t}{\log X} \Big)+ \frac{\Gamma'}{\Gamma} \Big( \frac 14+ \frac{k- 1}4 +\frac{\pi i t}{\log X} \Big)\Big) \phi(t) \d t\\
-\frac 2{\Omega_k(N)}\sum_{f\in B^*_k(N)} \omega_f(N) \sum_{p,\nu} \frac{\alpha_f^{\nu}(p)+\beta_f^{\nu}(p) }{p^{\nu/2}} \widehat \phi \Big( \frac{\nu \log p}{\log X} \Big) \frac{\log p}{\log X}. \label{Eq:explicit formula}
\end{multline}
Here, $\alpha_f(p)$ and $\beta_f(p)$ are the local coefficients of the $L$-function
$$L(s,f)= \sum_{n\geq 1} \frac{\lambda_f(n)}{n^s}=\prod_{p} \Big( 1-\frac {\alpha_f(p)}{p^s}\Big)^{-1}\Big( 1-\frac {\beta_f(p)}{p^s}\Big)^{-1}; $$
in particular, for $p\nmid N$ we have that $|\alpha_f(p)|=|\beta_f(p)|=1$, and for $p\mid N$ we use the convention that $\beta_f(p)=0$. 

\end{lemma}
\begin{proof}
For $f \in B_k^*(N)$, the formula\footnote{We fixed a minor typo in the arguments of the gamma factors.}  \cite[(4.11)]{ILS} reads
\begin{multline}\label{Eq:single explicit formula}
\sum_{\gamma_f} \phi\Big( \gamma_f \frac{\log X}{2\pi} \Big) = \widehat{\phi}(0) \frac{\log\frac N{\pi^2}}{\log X}+\frac 1{\log X} \int_{\mathbb R} \Big( \frac{\Gamma'}{\Gamma} \Big( \frac 14+ \frac{k+ 1}4 +\frac{\pi i t}{\log X} \Big)+ \frac{\Gamma'}{\Gamma} \Big( \frac 14+ \frac{k- 1}4 +\frac{\pi i t}{\log X} \Big)\Big) \phi(t) \d t\\
-2 \sum_{p,\nu} \frac{\alpha_f^{\nu}(p)+\beta_f^{\nu}(p) }{p^{\nu/2}} \widehat \phi \Big( \frac{\nu \log p}{\log X} \Big) \frac{\log p}{\log X}.
\end{multline}
Next, we sum against the weight $\omega_f(N)$ to obtain the desired formula. 
\end{proof}

We now introduce the Petersson formula, which involves the Kloosterman sum
$$ S(m,n;c) := \underset{x \bmod c}{{\sum}^*} e\Big( \frac{mx+n\overline{x}}{c}\Big)  $$
and the Bessel functions $J_k$.
\begin{lemma}
	\label{lemma:besselbounds}
	Let $k\in \mathbb N$.	We have the bound
	$$ J_{k-1}(x)\ll \min\Big( \frac 1{(k-1)!} \Big(\frac{x}{2}\Big)^{k-1}, x^{-\frac 14} (|x-k+1|+k^{\frac 13})^{-\frac 14}\Big).  $$
\end{lemma}

\begin{proof}
	This is~\cite[(2.11)', (2.11)'']{ILS}, which follows immediately from bounds in \cite{Wa,K}, see \cite[Lemma 3.2]{DFS}.
\end{proof}

\begin{lemma}[Petersson Formula]
\label{Petersson N Prime}
	Assuming that $N$ is prime, $(m,N)=1$, $(n,N^2) \mid N$, and $k$ is an even integer, we have the formula
	\begin{align}
	\nonumber &\sum_{f\in B_k^*(N)}  \omega_f(N) \lambda_f(m)\lambda_f(n) \\
	%		  &= \delta(m,n) + 2\pi i^k    \sum_{c \equiv 0 \bmod N} c^{-1} S(m,n,c) J_{k-1}(4\pi \sqrt{mn}/c) \notag  \\& -\frac{\delta(m,n)}{N} - \frac{2\pi i^k}N    \sum_{c } c^{-1} S(m,n,c) J_{k-1}(4\pi \sqrt{mn}/c)  + O\Big(  N^{-2}(kmn)^{\eps} \Big) \label{Petersson N Prime two terms}\\
	&= \delta(m,n) + 2\pi i^k    \sum_{c \equiv 0 \bmod N} c^{-1} S(m,n;c) J_{k-1}(4\pi \sqrt{mn}/c)  + O_{\eps}\Big(  (N\nu((n,N)))^{-1}(kmn)^{\eps} \Big), 
\end{align}
where 
$ \nu(t):=t\prod_{p\mid t} (1+\frac 1p)$.
\end{lemma}
\begin{proof}
	We combine \cite[Proposition 2.8]{ILS} with \cite[Proposition 2.1]{ILS}, in the case where $N$ is prime. In their notation, thanks to \cite[Lemma 2.5]{ILS}, we have that
	\begin{align}
		&	  \sum_{f\in B_k^*(N)}  \omega_f(N) \lambda_f(m)\lambda_f(n)= \frac{12}{(k-1)N} \Delta^*_{k,N}(m,n) \notag \\
		& = \sum_{LM=N} \frac{\mu(L)}{L\nu((n,L))}\sum_{\ell \mid L^{\infty}} \ell^{-1} \Delta_{k,M}(\ell^2 m,n) \notag \\
		&= \delta(m,n) + 2\pi i^k    \sum_{c \equiv 0 \bmod N} c^{-1} S(m,n;c) J_{k-1}(4\pi \sqrt{mn}/c)  -   \frac{1}{N\nu((n,N))}\sum_{\ell \mid N^{\infty}} \ell^{-1} \Delta_{k,1}(\ell^2m,n).\label{equation lemma ie}
	\end{align}
The first equality above is a restatement of the definition of $\Delta^*_{k,N}(m,n)$  \cite[(2.53), (2.54)]{ILS}, and moreover we recall (see~\cite[(2.7)]{ILS}) that $\Delta_{k,N}(m,n) = \sum_{f\in B_k(N)}  \omega_f(N) \lambda_f(m)\lambda_f(n)$ denotes the sum over all forms in an orthogonal basis $B_k(N)$ of the space of cusp forms of level $N$ and weight $k$ (including the old-forms).
To estimate the contribution of the last term on the right-hand side of~\eqref{equation lemma ie}, we use Deligne's bound $\lambda_f(n) \ll_\eps n^{\eps}$, which combined with the bounds $|B_k(1)| \ll k$ and $\omega_f(N) \ll_{\eps} (kN)^{-1+\eps}$ yields the bound $\Delta_{k,1}(\ell^2m,n) =O_{\eps}((\ell mnk)^{\eps})$, resulting in the claimed estimate.
\end{proof}

We continue by estimating the sum over $c$ in Lemma~\ref{Petersson N Prime}.

\begin{lemma}[Estimated Petersson Formula]
	\label{lemma:peterssonbound}
	Let $k$ be a fixed even integer.
	If $N$ is prime, $N^2\nmid n$ and $(m,N)=1$,
	we have
	$$  \sum_{f\in B_k^*(N)}  \omega_f(N) \lambda_f(m)\lambda_f(n)=\delta(m,n)+ O_{k,\eps}\Big(   (n,N)^{-\frac 12}  N^{-1 + \eps}(mn)^{\frac14 +\eps} \Big). $$
	If in addition $mn \leq N^2$, then we have the bound
	$$  \sum_{f\in B_k^*(N)}  \omega_f(N) \lambda_f(m)\lambda_f(n)=\delta(m,n)+ O_{k,\eps}\Big((N(n,N))^{-1}(mn)^{\eps} +  (n,N)^{-\frac 12 }  N^{-k+\frac12 +\eps} (mn)^{\frac{k-1}{2}+\eps} \Big). $$
\end{lemma}
\begin{proof}
By Lemma~\ref{Petersson N Prime}, as $N$ is prime, we have
	\begin{multline*}
		  \sum_{f\in B_k^*(N)}  \omega_f(N) \lambda_f(m)\lambda_f(n)= \delta(m,n) + 2\pi i^k    \sum_{c \equiv 0 \bmod N} c^{-1} S(m,n;c) J_{k-1}(4\pi \sqrt{mn}/c)  \\+ O_{k,\eps}\Big(  (N(n,N))^{-1}(mn)^{\eps} \Big). 
\end{multline*}
Using Weil's bound in the form~\cite[(2.13)]{ILS}, together with Lemma~\ref{lemma:besselbounds}, we have
\begin{align*}
	\sum_{c \equiv 0 \bmod N} c^{-1} S(m,n;c) J_{k-1}(4\pi \sqrt{mn}/c)
&\ll_{\eps,k} \sum_{c \equiv 0 \bmod N} c^{-\frac12 +\eps} \Big(\frac{\sqrt{mn}}{c}\Big)^{k-1} \frac{(m,n,c)}{(m,c)^\frac 12+(n,c)^\frac 12}  \\ 
&\ll  \frac{(m,n,N)}{ (m,N)^{\frac 12}+(n,N)^{\frac 12} }  N^{-k+\frac 12+\eps}  (mn)^{\frac{k-1}{2}}  \sum_{b \geq 1}  (m,n,b)^{\frac 12} b^{\frac 12-k+\eps} \\
& \ll_\eps  \frac{(m,n,N)}{ (m,N)^{\frac 12}+(n,N)^{\frac 12} }  N^{-k+\frac12 +\eps} (mn)^{\frac{k-1}{2}+\eps}, 
\end{align*}
which concludes the proof in the case $mn\leq N^2$.
In case $N^2<mn$, we cut the sum over $c$ and apply Weil's bound:
\begin{align*}
&	\sum_{c \equiv 0 \bmod N} c^{-1} S(m,n;c) J_{k-1}(4\pi \sqrt{mn}/c) \\
	&\ll_{\eps,k} \sum_{\substack{c \equiv 0 \bmod N \\ c> \sqrt{mn}}} c^{-\frac12 +\eps} \Big(\frac{\sqrt{mn}}{c}\Big)^{k-1} \frac{(m,n,c)}{(m,c)^\frac 12+(n,c)^\frac 12}  
	+ \sum_{\substack{c \equiv 0 \bmod N \\ c< \sqrt{mn}}} c^{-\frac12 +\eps} \Big(\frac{\sqrt{mn}}{c}\Big)^{-\frac12} \frac{(m,n,c)}{(m,c)^\frac 12+(n,c)^\frac 12}  \\
	&\ll \frac{(m,n,N)}{ (m,N)^{\frac 12}+(n,N)^{\frac 12} } N^{-1+\eps}(mn)^{\frac14 + \eps},
\end{align*}
which gives the desired bound.
\end{proof}

We now apply the Hecke relations and discard certain prime powers from the expression \eqref{Eq:explicit formula}.

\begin{lemma}\label{Lemma:prime powers}
Let $\phi$ be an even Schwartz test function for which $\sigma= \sup({\rm supp} (\widehat \phi)) <\infty$. For $N$ running through the set of prime numbers, we have that
\begin{align}
\Dstar = \widehat{\phi}(0) \frac{\log\frac N{\pi^2}}{\log X}+ \frac 1{\log X} \int_{\mathbb R} \Big(  \frac{\Gamma'}{\Gamma} \Big( \frac 14+ \frac{k+ 1}4 +\frac{\pi i t}{\log X} \Big)+ \frac{\Gamma'}{\Gamma} \Big( \frac 14+ \frac{k- 1}4 +\frac{\pi i t}{\log X} \Big)\Big) \phi(t) \d t \notag\\
+2\sum_{p \nmid N} \frac 1p \widehat \phi \Big( \frac{2 \log p}{\log X} \Big) \frac{\log p}{\log X}-\frac 2{\Omega_k(N)}\sum_{f\in B^*_k(N)} \omega_f(N) \sum_{\substack{p \nmid N \\ \nu \geq 1}} \frac{\lambda_f(p^\nu)}{p^{\nu/2}} \widehat \phi \Big( \frac{\nu\log p}{\log X} \Big) \frac{\log p}{\log X} + O_{k,\eps}(N^{-1+\eps}).% + N^{\frac{\sigma}{2}(k-2) -k + \frac12 + \eps}). 
\label{Eq:explicit formula without prime powers}
\end{align}
\end{lemma}
\begin{proof}
	Note that if $p = N$ then $\lambda_f(p^{\nu})=\lambda_f(p)^{\nu}$, and moreover $\lambda_f(p) = O(p^{-\frac12})$ \cite[Theorem~4.6.17]{Miyake}, so the contribution of this term to the prime sum in~\eqref{Eq:explicit formula} is $O_k(N^{-1})$.	
Next, by the Hecke relations, the sum over prime numbers not dividing $N$ in~\eqref{Eq:explicit formula} is equal to
\begin{align*}
&-\frac 2{\Omega_k(N)}\sum_{f\in B^*_k(N)} \omega_f(N) \sum_{\substack{p, \nu \\ p\nmid N}} \frac{ \lambda_f(p^{\nu}) }{p^{\nu/2}} \widehat \phi \Big( \frac{\nu \log p}{\log X} \Big) \frac{\log p}{\log X}  \\
&+\frac 2{\Omega_k(N)}\sum_{f\in B^*_k(N)} \omega_f(N) \sum_{\substack{p,\nu \geq 2 \\ p\nmid N}} \frac{ \lambda_f(p^{\nu-2}) }{p^{\nu/2}} \widehat \phi \Big( \frac{\nu \log p}{\log X} \Big) \frac{\log p}{\log X}.	
\end{align*} 
The contribution of the summands with $\nu=2$ in the second term are given by
\begin{align*}
	&\frac 2{\Omega_k(N)}\sum_{f\in B^*_k(N)} \omega_f(N) \sum_{\substack{p\nmid N}} \frac{ \lambda_f(1)}{p} \widehat \phi \Big( \frac{2 \log p}{\log X} \Big) \frac{\log p}{\log X}  =2\sum_{\substack{p\nmid N}} \frac{ 1}{p} \widehat \phi \Big( \frac{2 \log p}{\log X} \Big) \frac{\log p}{\log X}.  
\end{align*}
	The goal is now to estimate the contribution of the remaining terms, that is the terms with $\nu \geq 3$ and $p\neq N$.
	 From Lemma \ref{lemma:peterssonbound} and the prime number theorem, we see that
	
	\begin{align*}
	\frac 2{\Omega_k(N)}\sum_{f\in B^*_k(N)} \omega_f(N) \sum_{\substack{p \nmid N \\ \nu \geq 3}} \frac{ \lambda_f(p^{\nu-2})}{p^{\nu/2}}& \widehat \phi \Big( \frac{\nu \log p}{\log X} \Big) \frac{\log p}{\log X} \\
	&\ll_{k,\eps} N^{-1+\eps} \sum_{ p \leq X^{\sigma / 3 }} \frac{\log p}{p^{\frac12+2\eps}} \sum_{3 \leq \nu   \leq \sigma \frac{\log X}{\log p}  } p^{(-\frac14+\eps)\nu}  \\% + N^{-k+\frac12 +\eps} p^{(\nu-2)\frac{k-1}{2}} )  \\
%&\ll_\eps N^{-1+\eps} +   N^{-k+\frac12 +\eps}  \sum_{ \substack{ p \leq X^{\sigma}}} \frac 1{p^{k-1}} \sum_{3 \leq \nu   \leq \sigma \frac{\log X}{\log p}  } p^{\nu\frac{k-2}{2} } \\
&\ll N^{-1+\eps} \sum_{ p \leq X^{\sigma / 3 }} \frac{\log p}{p^{\frac54-\eps}} \ll N^{-1+\eps},
%&\ll_\eps N^{-1+\eps} +   N^{-k+\frac12 +\eps}  \sum_{p \leq X^{\sigma}}\frac{X^{\sigma\frac{k-2}{2} }}{p^{k-1}} \Big \lfloor \sigma \frac{\log X}{\log p} \Big \rfloor \\
%&\ll_{k,\eps}  N^{-1+\eps} +   N^{\frac{\sigma}{2}(k-2)-k+\frac12 +2\eps},
	\end{align*}
which gives the claimed result.
\end{proof}

We are now ready to apply the Petersson formula.
\begin{corollary}Let $\phi$ be an even Schwartz test function for which $\sigma= \sup({\rm supp} (\widehat \phi)) <2$. Then, for $N$ running through the set of prime numbers, we have that
	\begin{multline}
	\label{equation corollary expression kloosterman}
		\Dstar = \widehat{\phi}(0) + \frac{\phi(0)}{2} \\
	-\frac{4\pi i^k}{\Omega_k(N)}  \sum_{\substack{p \nmid N \\ \nu \geq 1}} \frac{1}{p^{\nu/2}} \widehat \phi \Big( \frac{\nu\log p}{\log X} \Big) \frac{\log p}{\log X} \sum_{c \equiv 0 \bmod N} c^{-1} S(p^{\nu},1;c) J_{k-1}(4\pi \sqrt{p^{\nu}}/c)
	+ O_k\Big(\frac{1}{\log X}\Big).
	\end{multline}
	\label{corollary expression kloosterman}
\end{corollary}

\begin{proof}
	By Lemma~\ref{Petersson N Prime}, we have
	\begin{multline*}
		-\frac 2{\Omega_k(N)}\sum_{f\in B^*_k(N)} \omega_f(N) \sum_{\substack{p \nmid N \\ \nu \geq 1}}\frac{\lambda_f(p^{\nu})}{p^{\nu/2}} \widehat \phi \Big( \frac{\nu\log p}{\log X} \Big) \frac{\log p}{\log X}
		\\=   -\frac{4\pi i^k}{\Omega_k(N)}  \sum_{\substack{p \nmid N \\ \nu \geq 1}} \frac{1}{p^{\nu/2}} \widehat \phi \Big( \frac{\nu\log p}{\log X} \Big) \frac{\log p}{\log X} \sum_{c \equiv 0 \bmod N} c^{-1} S(p^\nu,1;c) J_{k-1}(4\pi \sqrt{p^\nu}/c) + O_{k,\eps}(N^{\frac{\sigma}2 -1 + \eps}),
	\end{multline*}
which, combined with Lemma~\ref{Lemma:prime powers} and \cite[Lemmas~2.2 and 4.1]{DFS}, gives the desired result.
\end{proof}

\section{The prime sum in terms of zeros of Dirichlet $L$-functions}
\label{Section:Prime sum}

In the previous section we reduced the problem of estimating the one-level density to that of bounding sums over primes containing averages of Kloosterman sums. We will now evaluate these averages using Dirichlet characters, which will allow us to relate our problem to zeros of Dirichlet $L$-functions.

\begin{lemma}
\label{lemma dirichlet}
Let $\phi$ be an even Schwartz test function for which $\sigma= \sup({\rm supp} (\widehat \phi)) <2$. For $N$ running through the set of prime numbers, we have that
	\begin{multline}
	\label{equation lemma dirichlet}
		\Dstar =  \widehat{\phi}(0) + \frac{\phi(0)}{2}  \\
	-\frac{4\pi i^k}{\Omega_k(N)} \frac{1}{\log X}   \sum_{\substack{c \equiv 0 \bmod N \\ c < N^{1+ \frac{1}{2k-3}}}}  \frac{1}{c\varphi(c)}\sum_{\substack{\chi \bmod c\\ \chi \neq \chi_0}}\tau(\bar\chi)^2
 \sum_{\substack{p \nmid c \\ \nu \geq 1}} \frac{\log p}{p^{\nu/2}} \widehat \phi \Big( \frac{\nu\log p}{\log X} \Big) \chi(p^\nu) J_{k-1}(4\pi \sqrt{p^\nu}/c)   + O_k\Big(  \frac{1}{\log X}\Big)
.
	\end{multline}
\end{lemma}
\begin{proof}

For $p\nmid c$, the Kloosterman sums satisfy
\begin{align*}
	S(p^\nu,1;c) = \frac{1}{\varphi(c)}\sum_{a \bmod c}S(a,1;c) \sum_{\chi \bmod c}\bar\chi(a) \chi(p^\nu); \qquad  \sum_{a\bmod c} \overline{\chi}(a)S(a,1;c) = \tau(\overline \chi)^2 .
\end{align*}
We will 
substitute these expressions in Corollary~\ref{corollary expression kloosterman}. Let $C= N^{1+ \frac{1}{2k-3}}$.  Bounding the contribution of the primes dividing $c$ using \cite[(2.13)]{ILS} and Lemma~\ref{lemma:besselbounds}, the sum on the right-hand side of~\eqref{equation corollary expression kloosterman} is equal to
\begin{align*}
&-\frac{4\pi i^k}{\Omega_k(N)} \sum_{c \equiv 0 \bmod N} c^{-1} \sum_{\substack{p \nmid c \\ \nu \geq 1}} \frac{1}{p^{\nu/2}} \widehat \phi \Big( \frac{\nu\log p}{\log X} \Big) \frac{\log p}{\log X}   S(p^\nu,1;c) J_{k-1}(4\pi \sqrt{p^\nu}/c) + O_{k,\eps}\big(N^{(\sigma -2)\frac{k-2}{2}-\frac32 + \eps}\big) \\
&= -\frac{4\pi i^k}{\Omega_k(N)} \sum_{\substack{c \equiv 0 \bmod N \\ c < C}} c^{-1}\sum_{\substack{p \nmid c \\ \nu \geq 1}}\frac{1}{p^{\nu/2}} \widehat \phi \Big( \frac{\nu\log p}{\log X} \Big) \frac{\log p}{\log X}   S(p^\nu,1;c) J_{k-1}(4\pi \sqrt{p^\nu}/c) + O_{k,\eps}\big( N^{-\frac{3}2 +\eps} + N^{\frac{k}{2}(\sigma -2) +\eps}\big) \\
&= -\frac{4\pi i^k}{\Omega_k(N)} \frac{1}{\log X}   \sum_{\substack{c \equiv 0 \bmod N \\ c < C}}  \frac{1}{c\varphi(c)}\sum_{a \bmod c}\sum_{\chi \bmod c}\bar\chi(a)S(a,1;c) 
 \sum_{\substack{p \nmid c \\ \nu \geq 1}} \frac{\log p}{p^{\nu/2}} \widehat \phi \Big( \frac{\nu\log p}{\log X} \Big) \chi(p^{\nu}) J_{k-1}(4\pi \sqrt{p^\nu}/c) \\ & \qquad+ O_{k,\eps}\big(N^{-\frac{3}2 +\eps}  +N^{\frac{k}{2}(\sigma -2) +\eps}  \big) \\
 &= -\frac{4\pi i^k}{\Omega_k(N)} \frac{1}{\log X}   \sum_{\substack{c \equiv 0 \bmod N \\ c < C}}  \frac{1}{c\varphi(c)}\sum_{\substack{\chi \bmod c\\ \chi \neq \chi_0}}\tau(\bar\chi)^2
 \sum_{\substack{p \nmid c \\ \nu \geq 1}} \frac{\log p}{p^{\nu/2}} \widehat \phi \Big( \frac{\nu\log p}{\log X} \Big) \chi(p^\nu) J_{k-1}(4\pi \sqrt{p^\nu}/c) \\ & \qquad+ O_{k,\eps}\big( N^{-\frac{3}2 +\eps} + N^{\frac{k}{2}(\sigma -2) + \eps} +  N^{k(\frac{\sigma}2 -1)-1+\eps}\big),
\end{align*} 
where the error terms account for the contributions of terms with $p\mid c$, terms with $c \geq C$,  and the principal character.
\end{proof}

In the next lemma we express the third term on the right-hand side of~\eqref{equation lemma dirichlet} as a contour integral.
\begin{lemma}
\label{lemma dirichlet 2}
Let $\phi$ be an even Schwartz test function for which $\sigma= \sup({\rm supp} (\widehat \phi)) <2$. For $N$ running through the set of prime numbers, we have that
	\begin{multline}
	\label{equation lemma dirichlet 2}
		\Dstar = \widehat{\phi}(0) + \frac{\phi(0)}{2} 	+ O_k\Big(\frac{1}{\log X}\Big)
	\\  + O_k\bigg(   \sum_{\substack{c \equiv 0 \bmod N \\ c < N^{1+ \frac{1}{2k-3}}}}  \frac{1}{c\varphi(c)} \sum_{1\neq d\mid c} d \underset{\chi \bmod d}{{\sum}^*}   \bigg|
  \int_{(2)}  \frac{L'(s+\frac12,\chi)}{L(s+\frac12,\chi)}\Psi_{\phi,X,c,k}(s) \d s \bigg|    \bigg),
	\end{multline}
where the star over the sum means that we are summing over primitive characters, and where
\begin{equation}
\label{equation definition Psi}
\Psi_{\phi,X,c,k}(s) = \frac{1}{\log X} \int_0^{\infty} x^{s-1}\widehat \phi \Big( \frac{\log x}{\log X} \Big)  J_{k-1}(4\pi \sqrt{x}/c) \d x. 
\end{equation}
\end{lemma}

\begin{proof}
We will once more use the shorthand $C= N^{1+ \frac{1}{2k-3}}$.
Note that the sum on the right-hand side of~\eqref{equation lemma dirichlet} is equal to
$$ -\frac{4\pi i^k}{\Omega_k(N)} \frac{1}{\log X}   \sum_{\substack{c \equiv 0 \bmod N \\ c < C}}  \frac{1}{c\varphi(c)}\sum_{\substack{\chi \bmod c\\ \chi \neq \chi_0}}\tau(\bar\chi)^2
 \sum_{\substack{p \nmid c \\ \nu \geq 1}} \frac{\log p}{p^{\nu/2}} \widehat \phi \Big( \frac{\nu\log p}{\log X} \Big) \chi^*(p^\nu) J_{k-1}(4\pi \sqrt{p^\nu}/c),$$
 where $\chi^*$ is the primitive character modulo $c^*$ inducing $\chi$. Hence, this sum is 

\begin{align*}
 &\ll_k \frac{1}{\log X}   \sum_{\substack{c \equiv 0 \bmod N \\ c < C}}  \frac{1}{c\varphi(c)}\sum_{\substack{\chi \bmod c\\ \chi \neq \chi_0}}c^* \bigg|
 \sum_{\substack{p \nmid c \\ \nu \geq 1}} \frac{\log p}{p^{\nu/2}} \widehat \phi \Big( \frac{\nu\log p}{\log X} \Big) \chi^*(p^\nu) J_{k-1}(4\pi \sqrt{p^\nu}/c) \bigg| \\
 &= \frac{1}{\log X}   \sum_{\substack{c \equiv 0 \bmod N \\ c < C}}  \frac{1}{c\varphi(c)} \sum_{1\neq d\mid c}d\underset{\chi \bmod d}{{\sum}^*}   \bigg|
 \sum_{\substack{p \nmid c \\ \nu \geq 1}}\frac{\log p}{p^{\nu/2}} \widehat \phi \Big( \frac{\nu\log p}{\log X} \Big) \chi(p^\nu) J_{k-1}(4\pi \sqrt{p^\nu}/c) \bigg|. 
\end{align*}

The next step will be to apply Mellin inversion to the sum inside the absolute values. Doing so and interchanging the sum and integral thanks to absolute convergence (see Lemma~\ref{lemma:bound Psi} below), we obtain the identity
\begin{multline*}
 \frac{1}{\log X}\sum_{\substack{p \nmid c \\ \nu \geq 1}} \frac{\log p}{p^{\nu/2}} \widehat \phi \Big( \frac{\nu\log p}{\log X} \Big) \chi(p^\nu) J_{k-1}(4\pi \sqrt{p^\nu}/c)
 = - \frac{1}{2\pi i} \int_{(2)}  \frac{L'(s+\frac12,\chi)}{L(s+\frac12,\chi)}\Psi_{\phi,X,c,k}(s) \d s\\+O_\eps\Big(\frac{N^{\sigma(\frac k2-1)+\eps}}{c^{k-1}} \Big).
\end{multline*}
  The claimed estimate follows.
\end{proof}

We now establish analytic properties of the function $\Psi_{\phi,X,c,k}(s) $.

\begin{lemma}\label{lemma:bound Psi}
	The function $\Psi_{\phi,X,c,k}(s)$ defined in~\eqref{equation definition Psi} is entire and satisfies the bound
	\begin{align*}
		\Psi_{\phi,X,c,k}(s) \ll_k \frac{X^{\sigma \lvert\Re(s) + \frac{k-1}{2} \rvert}}{(\lvert s \rvert +1)^2 c^{k-1}}.
	\end{align*}
\end{lemma}

\begin{proof}
The fact that $\widehat \phi$ has compact support immediately implies that $\Psi_{\phi,X,c,k}(s)$ is entire.
	Assume now  that $\lvert s \rvert >1$.
	We change variables $u= \frac{\log x}{\log X}$ in the definition of $\Psi_{\phi,X,c,k}(s)$ and then integrate by parts twice:
	\begin{align*}
		\Psi_{\phi,X,c,k}(s) &= \int_{\mathbb{R}} X^{us} \widehat\phi(u) J_{k-1}(4 \pi X^{\frac{u}{2}}/c) \d u \\
		&= -\int_{\mathbb{R}} \frac{X^{us}}{s \log X} \Big( \widehat\phi'(u) J_{k-1}(4 \pi X^{\frac{u}{2}}/c) + \widehat\phi(u) \frac{2 \pi}{c} X^{\frac{u}{2}}\log X J'_{k-1}(4 \pi X^{\frac{u}{2}}/c)\Big) \d u \\
		&= \int_{\mathbb{R}} \frac{X^{us}}{(s \log X)^2} \Big( \widehat\phi''(u) J_{k-1}(4 \pi X^{\frac{u}{2}}/c)   + \widehat\phi'(u) \frac{4 \pi}{c} X^{\frac{u}{2}}\log X J'_{k-1}(4 \pi X^{\frac{u}{2}}/c) \\& 
		\hspace{1cm} +   \widehat\phi(u) \frac{\pi}{c} X^{\frac{u}{2}}(\log X)^2 J'_{k-1}(4 \pi X^{\frac{u}{2}}/c) + \widehat\phi(u) \frac{4 \pi^2}{c^2} X^{u}(\log X)^2 J''_{k-1}(4 \pi X^{\frac{u}{2}}/c)\Big) \d u.
	\end{align*}
		By the identity
$$2J'_k(x) =  J_{k-1}(x)-J_{k+1}(x)  $$ 
(see \cite[equation~(2) page~17]{Wa}) and Lemma~\ref{lemma:besselbounds}, we deduce the bound\footnote{Note that in the case $k=2$, we also use the identity $J_{-1}(x)=-J_1(x)$.}
\begin{align*}
	\Psi_{\phi,X,c,k}(s) &\ll_k \int_{-\sigma}^{\sigma} \frac{X^{u \Re(s)}}{\lvert s\rvert^2}  \Big( \frac{X^{\frac{u}{2}}}{c}\Big)^{k-1}  \d u 
	\ll \frac{X^{\sigma \lvert\Re(s) + \frac{k-1}{2} \rvert}}{\lvert s \rvert^2 c^{k-1}} .
\end{align*}
Finally, for $\lvert s \rvert\leq1$ the bound $\Psi_{\phi,X,c,k}(s)  \ll \frac{X^{\sigma \lvert\Re(s) + \frac{k-1}{2} \rvert}}{c^{k-1}}$ follows directly from the definition.
\end{proof}

The next step will be to move the contour of integration to the left in the integral appearing in~\eqref{equation lemma dirichlet 2}. 

\begin{proposition}\label{corollary:sum over zeros}
	Let $\phi$ be an even Schwartz test function for which $\sigma= \sup({\rm supp} (\widehat \phi)) <2$.
	For $k\geq 2$ a fixed even integer, and for $N$ running through the set of prime numbers,
we have 
\begin{multline}
	\Dstar = \widehat{\phi}(0) + \frac{\phi(0)}{2}  	+ O_k\Big(\frac{1}{\log X}\Big)
	 \\ + O_{k,\eps}\bigg( N^{\sigma(\frac{k}{2} -1) + \eps} \sup_{N \leq D <  N^{1+ \frac{1}{2k-3}}} \sup_{\frac12\leq\beta<1} \sup_{1\leq T \leq N^5 }  \frac{N^{\sigma\beta}}{D^{k} T^2} \sum_{ \substack{d \sim D \\ d\equiv 0 \bmod N }} \underset{\chi \bmod d}{{\sum}^*} 
	\sum_{\substack{\rho_\chi \\ \beta\leq\Re(\rho_\chi) < \beta+ \frac{1}{\log N} \\ T-1\leq \lvert\Im(\rho_\chi)\rvert<2T } } 1\bigg)
,
\end{multline}
where $\rho_{\chi}$ is running over the non-trivial zeros of  the Dirichlet $L$-function~$L(s,\chi)$ and $d\sim D$ means $\frac D2<d \leq D$.
\end{proposition}

\begin{proof}

For an integer $d>1$, we let $\chi \bmod d$ be a primitive character.
We will pull the contour of integration to the left in the integral
$$\frac{1}{2\pi i} \int_{(2)}  \frac{L'(s+\frac12,\chi)}{L(s+\frac12,\chi)}\Psi_{\phi,X,c,k}(s) \d s, $$
which appears in Lemma~\ref{lemma dirichlet 2}. Moreover, following~\cite[\S 19]{D} and applying Lemma~\ref{lemma:bound Psi}, this integral is easily shown to be equal to 
$$ \sum_{\rho_\chi} \Psi_{\phi,X,c,k}\big(\rho_\chi - \tfrac12\big) + \Psi_{\phi,X,c,k}\big(-\tfrac12\big)1_{\chi(-1)=1} + \frac{1}{2\pi i} \int_{(-1)}  \frac{L'(s+\frac12,\chi)}{L(s+\frac12,\chi)}\Psi_{\phi,X,c,k}(s) \d s.$$
Arguing as in~\cite[\S 19]{D}, we can show that on the line $\Re(s)=-1$ we have the bound
$$\frac{L'(s+\frac12,\chi)}{L(s+\frac12,\chi)} \ll \log (d\lvert s\rvert).$$ 
Combining this with Lemma~\ref{lemma:bound Psi}, we  deduce that
\begin{align*}
	\frac{1}{2\pi i} \int_{(-1)}  \frac{L'(s+\frac12,\chi)}{L(s+\frac12,\chi)}\Psi_{\phi,X,c,k}(s) \d s
	&\ll_k \int_{(-1)} \log (d\lvert s\rvert) \frac{X^{\sigma \lvert \frac{k-3}{2}\rvert}}{\lvert s\rvert^2 c^{k-1}} \d s
	\ll \frac{X^{\sigma \lvert \frac{k-3}{2}\rvert}}{ c^{k-1}} \log d.
\end{align*}
The total contribution of this term to the second error term in~\eqref{equation lemma dirichlet 2} is $\ll_{k,\eps} N^{\sigma \lvert \frac{k-3}{2}\rvert - k +1 +\eps}$ (recall that $X=k^2N$), which is admissible for $\sigma < 2$ and any even $k\geq 2$. 
Using Lemma~\ref{lemma:bound Psi}, we see that the contribution of the terms $\Psi_{\phi,X,c,k}\big(-\tfrac12\big)$ is $O_{k,\eps}(N^{\sigma(\frac{k}2 -1) - k +1 +\eps})$ which is also  admissible.
We are left with the term involving non-trivial zeros of $L(s,\chi)$ whose contribution to the second error term in~\eqref{equation lemma dirichlet 2} is (we set $ C=N^{1+ \frac{1}{2k-3}}$)
\begin{align*}
	&  \ll_k \sum_{\substack{c \equiv 0 \bmod N \\ c < C}}  \frac{1}{c\varphi(c)} \sum_{1\neq d\mid c}d\underset{\chi \bmod d}{{\sum}^*} 
		 \sum_{\rho_\chi} |\Psi_{\phi,X,c,k}\big(\rho_\chi - \tfrac12\big)| \\ 	&\ll_k
	 N^{\sigma(\frac{k}{2} -1)}\sum_{\substack{c \equiv 0 \bmod N \\ c < C}}  \frac{1}{c^{k}\varphi(c) }\sum_{1\neq d\mid c}d \underset{\chi \bmod d}{{\sum}^*} 
	\sum_{\rho_\chi} \frac{N^{\sigma\Re(\rho_\chi)}}{(\lvert \rho_\chi  \rvert +1)^2},
\end{align*}	
by Lemma~\ref{lemma:bound Psi}. This expression is
\begin{align*}
	&\ll_\eps N^{\sigma(\frac{k}{2} -1) -k -1+\eps}\sum_{ c' < C/N}  \frac{1}{c'^{k+1} }\sum_{\substack{d\mid c'N \\ d\neq 1}} d \underset{\chi \bmod d}{{\sum}^*} 
	\sum_{\rho_\chi} \frac{N^{\sigma\Re(\rho_\chi)}}{(\lvert \rho_\chi  \rvert +1)^2} \\
	&\ll N^{\sigma(\frac{k}{2} -1) -k -1+\eps}\sum_{\substack{ 1<d< C}} d \sum_{ \substack{c' < C/N\\ c' \equiv 0 \bmod \frac{d}{(d,N)}}}  \frac{1}{c'^{k+1} }\underset{\chi \bmod d}{{\sum}^*} 
	\sum_{\rho_\chi} \frac{N^{\sigma\Re(\rho_\chi)}}{(\lvert \rho_\chi  \rvert +1)^2} \\
	&\ll N^{\sigma(\frac{k}{2} -1) -k -1+\eps}\sum_{\substack{ 1<d< C}}d \Big(\frac{d}{(d,N)}\Big)^{-k-1}\underset{\chi \bmod d}{{\sum}^*} 
	\sum_{\rho_\chi} \frac{N^{\sigma\Re(\rho_\chi)}}{(\lvert \rho_\chi  \rvert +1)^2} \\
	&\ll N^{\sigma(\frac{k}{2} -1) +\eps}\sum_{\substack{ d< C \\ d \equiv 0\bmod N}} d^{-k}\underset{\chi \bmod d}{{\sum}^*} 
	\sum_{\rho_\chi} \frac{N^{\sigma\Re(\rho_\chi)}}{(\lvert \rho_\chi  \rvert +1)^2} + O_{\eps}\big(N^{\sigma\frac k2-k-1 +\eps}\big).
\end{align*}

The sum over $\rho_\chi$ can be truncated at $|\Im(\rho_\chi)| \leq N^{5}$ modulo an admissible error term. The final step is to decompose the sum over zeros into subsums in which $\Re(\rho_\chi)$ lies in a short vertical band of size $\frac1{\log N}$ (on which the function $N^{\sigma\Re(\rho_\chi)}$ is essentially constant), and in which $\Im(\rho_\chi)$ lies in a dyadic interval. Note that by monotonicity and by symmetry of the zeros about the line $\Re(s)=\frac 12$, we may restrict the sum over zeros to the $\rho_\chi$ for which $\Re(\rho_\chi) \geq \frac 12$. Finally, we cut the sum over $d$ dyadically. Once this is done, we obtain the claimed estimate.
\end{proof}

\section{Proof of Theorem~\ref{theorem main}}\label{Section:zero density}

The object of this section is to state and apply a zero-density estimate on Dirichlet $L$-functions, which will be the last step in the proof of Theorem~\ref{theorem main}. For a Dirichlet character $\chi$, we denote $$N(\beta,T,\chi):= \#\lbrace \rho_\chi : \Re(\rho_\chi) \geq  \beta, \lvert\Im(\rho_\chi)\rvert <T , L(\rho_\chi,\chi) =0\rbrace.$$

 \begin{theorem}
 	\label{theorem zero density IK}
 	Fix $\eps >0$. In the range $\frac 12+\eps\leq \beta\leq 1$ and for $h \in \mathbb N$, $Q\geq 1$, we have the bound
 	$$ \sum_{\substack{q\leq Q \\ (q,h)=1}}  \underset{\psi \bmod q}{{\sum}^*}
 	\hspace{.2cm}  \sum_{\xi \bmod h} N(\beta,T,\psi\xi) \ll_\eps \big((hQT)^{(2+\eps)(1-\beta)}+(hQ^2T)^{(1-\beta) \min(\frac 3{2-\beta},\frac 3{3\beta-1})}\big)(\log hQT)^{O_{\eps}(1)}. $$

 \end{theorem}
 
 \begin{proof}
 This is a slightly modified version of~\cite[Theorem 10.4]{IK}. Going through the proof of this result, one can see that in the range $\frac 12+\eps\leq \beta\leq 1$, the implied constant in $\ll_\eps$ and that in the power of $(\log hQT)$ depend only on $\eps$ (rather than on $\beta$ itself).
 \end{proof}

We are now ready to prove our main theorem.

\begin{proof}[Proof of Theorem~\ref{theorem main}]

From Proposition~\ref{corollary:sum over zeros}, we see that it is sufficient to show that (we set $C=N^{1+ \frac{1}{2k-3}}$)
\begin{equation}
  N^{\sigma(\frac{k}{2} -1) + \eps} \sup_{N \leq D < C} \sup_{\frac12\leq\beta<1} \sup_{1\leq T \leq N^5 } \frac{N^{\sigma\beta}}{D^{k} T^2} \sum_{ \substack{d \sim D \\ d\equiv 0 \bmod N }}\underset{\chi \bmod d}{{\sum}^*} 
	\sum_{\substack{\rho_\chi \\ \beta\leq\Re(\rho_\chi) < \beta+ \frac{1}{\log N} \\ T-1\leq \lvert\Im(\rho_\chi)\rvert<2T } } 1=o_{N\rightarrow \infty}(1).
\label{equation sumtobound proof 1.1}
\end{equation}
	Note that since $D<N^2$, $N^2$ does not divide $d$ and thus we may write $d=fN$ with $(f,N)=1$.
Now, by the Chinese Remainder Theorem, every primitive character $\chi$  modulo $d$ where $d$ is a multiple of $N$ can be decomposed as $ \chi=\psi\xi$ where $\psi$ is primitive modulo $f$ and $\xi$ is primitive modulo $N$. As a result, the left-hand side of~\eqref{equation sumtobound proof 1.1} is 
$$ \ll_\eps  N^{\sigma(\frac{k}{2} -1) + \eps} \sup_{N \leq D < C} \sup_{\frac12\leq\beta<1} \sup_{1\leq T \leq N^5 } \frac{N^{\sigma\beta}}{D^{k} T^2} \sum_{ \substack{f \sim D/N \\ (f,N)=1 }} \underset{\psi \bmod f} {{\sum}^*}
\hspace{.2cm} \underset{\xi \bmod N}{{\sum}^*}
	N(\beta,2T,\psi\xi). $$

We first apply the Riemann--von Mangoldt theorem in the range $\frac 12 \leq \beta <\frac 12+\eps $ and deduce that 	
	$$  N^{\sigma(\frac{k}{2} -1) + \eps} \sup_{N \leq D < C} \sup_{\frac12\leq\beta < \frac 12+\eps } \sup_{1\leq T \leq N^5 } \frac{N^{\sigma\beta}}{D^{k} T^2} \sum_{ \substack{f \sim D/N \\ (f,N)=1 }} \underset{\psi \bmod f} {{\sum}^*}
\hspace{.2cm} \underset{\xi \bmod N}{{\sum}^*}
	N(\beta,2T,\psi\xi)\ll N^{\sigma( \frac k2-\frac 12 + \eps) +1-k +2\eps}, $$
	which is admissible for $\sigma < 2-9\eps$. Moreover, Theorem~\ref{theorem zero density IK} implies that 
\begin{multline}
 N^{\sigma(\frac{k}{2} -1) + \eps} \sup_{N \leq D < C} \sup_{\frac 12+\eps \leq\beta <1  }  \sup_{1\leq T \leq N^5 } \frac{N^{\sigma\beta}}{D^{k} T^2} \sum_{ \substack{f \sim D/N \\ (f,N)=1 }} \underset{\psi \bmod f} {{\sum}^*}
\hspace{.2cm} \underset{\xi \bmod N}{{\sum}^*}
	N(\beta,2T,\psi\xi) \\
 \ll_\eps  N^{\sigma(\frac{k}{2} -1) + 3\eps} \sup_{N \leq D < C} \sup_{\frac 12+\eps \leq\beta <1  }  \sup_{1\leq T \leq N^5 } \frac{N^{\sigma\beta}}{D^{k} T^{2}} \Big( (DT)^{2(1-\beta)}+(N^{-1}D^2T)^{(1-\beta) \min(\frac 3{2-\beta},\frac 3{3\beta-1})}\Big).
 \label{equation after zero density} 
\end{multline}

As $k\geq 2$, the supremum over $D$ is attained at $D=N$ and the supremum over $T$ is attained at $T=1$, thus, the error term above is
\begin{align*}
	&\ll N^{\sigma(\frac{k}{2} -1) + 3\eps} \sup_{\frac 12+\eps \leq\beta <1  }  N^{\sigma\beta-k} \Big(N^{2(1-\beta)} +N^{(1-\beta) \min(\frac 3{2-\beta},\frac 3{3\beta-1})} \Big) 
	\\
	&\leq  
	\sup_{\frac 12 \leq\beta <1  }    \Big( N^{\sigma(\beta +\frac k2-1)-k + 2 (1-\beta)+3\eps} +N^{\sigma(\beta+\frac k2-1)-k+(1-\beta)\min(\frac 3{2-\beta},\frac 3{3\beta-1})+3\eps} \Big) .
%	\\
%	&=   N^{\sigma(\frac{k}{2} -1) + \eps} + \sup_{\frac12\leq\beta<1} N^{\sigma\beta+(1-\beta) \min(\frac 3{2-\beta},\frac 3{3\beta-1})-\sigma(\frac{k}{2} -1) + \eps} .
\end{align*}
The first error term is admissible whenever
$$ \sigma< \inf_{\frac12 \leq \beta <1} \frac{2\beta +k-2}{ \beta+\frac k2 -1} = 2. $$
As for the second, it is admissible whenever
$$ \sigma< \inf_{\frac12 \leq \beta <1} \frac{k-(1-\beta)\min(\frac{3}{2-\beta},\frac 3{3\beta-1})}{ \beta+\frac k2 -1} =\Theta_k,$$
where we recall that $\Theta_k$ is defined in~\eqref{equation definition thetak}. The last equality follows from optimizing the two cases in the minimum separately. Note that for $k=2$, the infimum is attained at $\beta=\sqrt 3-1$, and for $k\geq 4$ it is attained at $\beta=\frac 34$.
\end{proof}

\begin{remark}
\label{remark grand}
As mentioned in the introduction, the Grand Density Conjecture 
$$ \sum_{\substack{q\leq Q \\ (q,k)=1}}\underset{\psi \bmod q}{{\sum}^*} \hspace{.2cm} \sum_{\xi \bmod k} N(\beta,T,\psi\xi) \ll (kQ^2T)^{2(1-\beta)}(\log kQT)^{O(1)} $$
(see~\cite[p.250]{IK}) implies the Katz--Sarnak conjecture in the full range ${\rm supp} (\widehat \phi) \subset (-2,2)$. Indeed, under this conjecture we can replace~\eqref{equation after zero density} by the expression
$$N^{\sigma(\frac{k}{2} -1) +2 \eps} \sup_{N \leq D < N^{1+ \frac{1}{2k-3}}} \sup_{\frac12\leq\beta<1} \sup_{1\leq T \leq N^5 } \frac{N^{\sigma\beta}}{D^{k} T^{2}} ( N^{-1}D^2T)^{2(1-\beta)} \ll\sup_{\frac12\leq\beta<1}  N^{(\sigma-2) (\beta+\frac k2-1) +2\eps},$$
which is clearly admissible when $\sigma <2$.
\end{remark}

\end{document}